\documentclass[12pt]{amsart}
\usepackage{fullpage,amssymb}
\newtheorem{theorem}{Theorem}[section]

\newtheorem{corollary}[theorem]{Corollary}
\newtheorem{proposition}[theorem]{Proposition}
\newtheorem*{notation*}{Notation}
\newtheorem*{p*}{Proposition~\ref{h.s.o.p}}
\theoremstyle{definition}
\newtheorem{definition}[theorem]{Definition}
\newtheorem{conjecture}[theorem]{Conjecture}

\newtheorem{remark}[theorem]{Remark}
\newcommand{\M}{{\operatorname{Mat}}}
\newcommand{\R}{{\operatorname{Rep}}}
\newcommand{\C}{{\mathbb C}}
\newcommand{\N}{{\mathbb N}}
\newcommand{\Z}{{\mathbb Z}}

\newcommand{\K}{K}
\newcommand{\Sl} {{\operatorname{SL} }}

\newcommand{\SL}{{\operatorname{SL}}}

\newcommand{\Hom}{\operatorname{Hom}}
\newcommand{\kar}{\operatorname{char}}

\newcommand{\Sym}{\operatorname{Sym}}
\newcommand{\GL}{\operatorname{GL}}
\newcommand{\Rep}{\operatorname{Rep}}
\newcommand{\SI}{\operatorname{SI}}
\newcommand{\I}{\operatorname{I}}

\newcommand{\Proj}{\operatorname{Proj}}

\usepackage{multirow}
\usepackage{amsmath}
\usepackage{comment}
\usepackage{tikz}
\usepackage{mathtools}
\usepackage{arydshln}
\usepackage{graphicx}

\title{Generating invariant rings of quivers in arbitrary characteristic}

\author{Harm Derksen and Visu Makam}

\thanks{The authors were supported by NSF grant DMS-1601229}

\begin{document}

\begin{abstract}
It is well known that the ring of polynomial invariants of a reductive group is finitely generated. However, it is difficult to give strong upper bounds on the degrees of the generators, especially over fields of positive characteristic. In this paper, we make use of the theory of good filtrations along with recent results on the null cone to provide polynomial bounds for matrix semi-invariants in arbitrary characteristic, and consequently for matrix invariants. Our results generalize to invariants and semi-invariants of quivers.
\end{abstract}

\maketitle

\section{Introduction}
Fix an algebraically closed field $\K$.

\subsection{Degree bounds on invariant rings}
For a rational representation $V$ of a reductive group $G$, the ring of polynomial invariants $K[V]^G$ is a finitely generated graded subalgebra of the coordinate ring $K[V]$ (see \cite{Haboush,Hilbert1,Hilbert2,Nagata}). Unfortunately, the proof is not constructive. While finding a minimal set of generators is perhaps too hard a question to answer, once could ask instead for a bound on the degree of generators.

\begin{definition}
$\beta(\K[V]^G)$ is defined as the smallest integer $D$ such that the homogeneous invariants of degree $\leq D$ generate the invariant ring $\K[V]^G$.
\end{definition}

The methods of Popov and Derksen give us a general method to obtain bounds in characteristic $0$, see \cite{Derksen1,DK,Popov1,Popov2}. Such a method does not exist in positive characteristic. Nevertheless, in the cases of invariants and semi-invariants of quivers, we are able to bring together the theory of good filtrations and our recent results on the null cone for matrix semi-invariants to obtain several strong bounds in arbitrary characteristic.

\subsection{Preliminaries on quivers}
A quiver is just a directed graph. Formally a quiver is a pair
$Q=(Q_0,Q_1)$, where $Q_0$ is a finite set of vertices and $Q_1$ is a finite set of arrows. For an arrow $a\in Q_1$ we denote its head and tail by $ha$ and $ta$ respectively.
 A representation $V$ of $Q$ over $K$ is a collection of finite dimensional $K$-vector spaces $V(i)$, $i \in Q_0$ together with a collection of $K$-linear maps $V(a):V(ta)\to V(ha)$, $a\in Q_1$. The dimension vector of $V$ is  $\alpha \in \N^{Q_0}$, where $\N = \{0,1,2\dots\}$, such that
$\alpha_i =\dim V(i)$ for all $x \in Q_0$. 
For a dimension vector $\alpha\in \N^{Q_0}$, fix a vector space $V(i)$ of dimension $\alpha_i$ at each vertex $i \in Q_0$. We define the representation space by:
$$\Rep(Q,\alpha)=\prod_{a\in Q_1} \Hom(V(ta),V(ha)).$$ 
 Consider the group $\GL(\alpha)=\prod_{i\in Q_0} \GL(V(i))$ and its subgroup $\Sl(\alpha)=\prod_{i \in Q_0}\SL(V(i))$.
The group $\GL(\alpha)$ 
acts on $\Rep(Q,\alpha)$ by:
$$
(A(i)\mid i \in Q_0)\cdot (V(a)\mid a\in Q_1)=(A(ha)V(a)A(ta)^{-1}\mid a\in Q_1).
$$
For $V\in \Rep(Q,\alpha)$, choosing a different basis means acting by the group $\GL(\alpha)$. The $\GL(\alpha)$-orbits in $\Rep(Q,\alpha)$ correspond to isomorphism classes of representations of dimension $\alpha$.
The group $\GL(\alpha)$ also acts (on the left)  on the ring $K[\Rep(Q,\alpha)]$ of polynomial functions on $\Rep(Q,\alpha)$ by
$$
A\cdot f(V)=f(A^{-1}\cdot V)
$$
where $f\in K[\Rep(Q,\alpha)]$, $V\in \Rep(Q,\alpha)$ and $A\in \GL(\alpha)$. We refer the reader to \cite{DW,DW2} for further details on quivers.

\subsection{Invariants of quivers}
The invariant ring $\I(Q,\alpha) = K[\R(Q,\alpha)]^{\GL(\alpha)}$ is called the ring of invariants for $Q$ with respect to the dimension vector $\alpha$.

If $Q$ is the $m$-loop quiver and we choose the dimension vector $n \in \N^{Q_0} = \N$, then $\R(Q,\alpha) = \M_{n,n}^m$ and $\GL(\alpha) = \GL_n$ acts on $\M_{n,n}^m$ by simultaneous conjugation. The ring of invariants $S(n,m): = K[\M_{n,n}^m]^{\GL_n}$ is commonly referred to as the ring of matrix invariants. 

In characteristic $0$, the work of Procesi and Razmyslov in the $70$s gives the following bound on the degree of generators, see \cite{Procesi,Razmyslov,Formanek}.

\begin{theorem} [Procesi-Razmyslov] \label{PR}
Let $\kar K = 0.$ Then the ring $S(n,m)$ is generated by invariants of degree $\leq n^2$.
\end{theorem}

In $2002$, explicit bounds in positive characteristic were shown by Domokos in \cite{Domokos2}. See also \cite{DKZ,Lopatin}.

\begin{theorem} [Domokos]
The ring $S(n,m)$ is generated by invariants of degree $O(n^7m^n)$
\end{theorem} 

Let $C_{n,m,K}$ denote the smallest integer $d$ such that all monomials of degree $d$ in the free non-unital associative algebra $K\langle x_1,x_2,\dots,x_m\rangle$ are contained in the T-ideal generated by $x_1^n$. The above results relied on finding bounds for the constant $C_{n,m,K}$, see \cite{Domokos,Domokos2,Lopatin}. 

Using radically different methods, we obtain bounds for $\beta(S(n,m))$ that are in fact polynomial in $n$ and $m$.

\begin{theorem} \label{mi.bds}
The ring $S(n,m)$ is generated by invariants of degree $\leq (m+1)n^4$.
\end{theorem}

 Given our result, and the intimate relationship between $\beta(S(n,m))$ and $C_{n,m,K}$, we make the following conjecture:

\begin{conjecture}
There is a polynomial in two variables $f(n,m)$ such that $C_{n,m,K} = O(f(n,m)).$ 
\end{conjecture}

Theorem~\ref{PR} was generalized to invariants of quivers by Le~Bruyn and Procesi in \cite{LP}. We give a similar generalization using the description of invariants given by Donkin in positive characteristic.

\begin{corollary} \label{inv.bds}
The ring $\I(Q,\alpha)$ is generated by invariants of degree $(M + 1)N^4$, where $M = |Q_1|$ and $N = \sum_{i \in Q_0} \alpha_i$.
\end{corollary}

\subsection{Semi-invariants of quivers}

The invariant ring $\SI(Q,\alpha)=K[\Rep(Q,\alpha)]^{\SL(\alpha)}$ is called the ring of semi-invariants. 

A multiplicative character of the group $\GL_\alpha$ is of the form
$$
\chi_\sigma:(A(i)\mid i\in Q_0)\in \GL_\alpha\mapsto \prod_{i\in Q_0}\det(A(i))^{\sigma(i)}\in K^\star,
$$
where $\sigma:Q_0\to \Z$ is called the weight of the character $\chi_\sigma$. Define 
$$\SI(Q,\alpha)_\sigma=\{f\in K[\Rep(Q,\alpha)]\mid \forall A\in \GL(\alpha),\ A\cdot f=\chi_\sigma(A) f\}.$$
Then we have $\SI(Q,\alpha)=\bigoplus_{\sigma}\SI(Q,\alpha)_\sigma$. 

For a given weight $\sigma$, we can consider the subring $\SI(Q,\alpha,\sigma)=\bigoplus_{d=0}^\infty \SI(Q,\alpha)_{d\sigma}$. For any weight $\sigma$, the projective variety $\Proj(\SI(Q,\alpha,\sigma))$, if nonempty,  is a moduli space for the $\alpha$-dimensional representations of the quiver $Q$. See \cite{King} for more details. 

For the $m$-Kronecker quiver, i.e, a quiver a with two vertices $x$ and $y$ and $m$ arrows going from $x$ to $y$ and dimension vector $\alpha = (n,n)$,  we have $\Rep(Q,\alpha) = \M_{n,n}^m$ and $\SL(\alpha) = \SL_n \times \SL_n$. The ring of semi-invariants $R(n,m) := K[\M_{n,n}^m]^{\SL_n \times \SL_n}$ is known as the ring of matrix semi-invariants. 

\begin{theorem} \label{msi}
 The ring $R(n,m)$ is generated by invariants of degree $\leq mn^4$.
\end{theorem}

The ring $\SI(Q,\alpha,\sigma)$ is graded where $\SI(Q,\alpha)_{d\sigma}$ is the degree $d$ part. Define $|\sigma|_{\alpha} = \frac{1}{2} |\sigma| \cdot~\alpha = \frac{1}{2} \sum_{i \in Q_0} |\sigma(i)| \alpha_i$

\begin{corollary} \label{si1}
Let $Q$ be a quiver with no oriented cycles. Then the ring $\SI(Q,\alpha,\sigma)$ is generated by invariants of degree $\leq mn^3$, where $n = |\sigma|_{\alpha}$ and $m=\sum_{x\in Q_0}\sum_{y\in Q_0} \sigma_+(x)b_{x,y}\sigma_-(y)$ where $b_{x,y}$ is the number of paths from $x$ to $y$.
\end{corollary}

The ring $\SI(Q,\alpha)$ has a weight space decomposition, and we give our bounds in terms of this weight space decomposition rather than the usual $\Z$-grading by total degree. 

\begin{theorem} \label{si2}
Let $Q$ be a quiver with no oriented cycles, $\alpha \in \Z^{Q_0}$ a dimension vector. Let $r$ denote the Krull dimension of $\SI(Q,\alpha)$, and let $|Q_0| = n$. The ring $\SI(Q,\alpha)$ is generated by semi-invariants of weights $\sigma$ with $$|\sigma|_\alpha \leq   \displaystyle \frac{3rn^2 ||\alpha||_1^{4n}}{128(n-1)^{4n-4}},$$
where $||\alpha||_1 = \sum_{i \in Q_0} |\alpha_i|$.
\end{theorem}

Note that $\dim(\SI(Q,\alpha)) \leq \dim \Rep(Q,\alpha)$, which depends on $Q_0$ and $Q_1$. We showed the above results in characteristic $0$ in \cite{DM,DM2}. In this paper, we show that the restrictions on characteristic can be removed.

\subsection{Separating invariants}
In characteristic $0$, Weyl's theorem on polarization of invariants essentially tells us that for a rational representation $V$ of a reductive group $G$, we have $\beta(K[V^{\oplus m}]^G) \leq \beta(\K[V^{\oplus \dim V}]^G)$, see \cite{KP}. Such a result is not true in positive characteristic, and in fact Domokos shows this explicitly for the case of matrix invariants, see \cite{Domokos}. However, an analagous statement holds for separating invariants in arbitrary characteristic.

\begin{definition}
A subset $S$ of $\K[V]^G$ is called a separating subset if for all $v,w \in V$ such that $f(v) \neq f(w)$ for some $f \in \K[V]^G$, there exists $g \in S$ such that $g(v) \neq g(w)$. Define $\beta_{sep}(K[V]^G)$ to be the smallest integer $D$ such that the invariants of degree $\leq D$ form a separating subset.
\end{definition}

In \cite{DKW}, Draisma, Kemper and Wehlau showed the following:

\begin{theorem}[Draisma-Kemper-Wehlau] \label{Weyl.sep}
We have $\beta_{\rm sep}(\K[V^{\oplus m}]^G) \leq \beta_{\rm sep}(\K[V^{\oplus dim V}]^G)$ for all $m \in \N$.
\end{theorem}

As a corollary, we can obtain bounds for separating invariants in the cases of $S(n,m)$ and $R(n,m)$ that are independent of $m$. 

\begin{corollary} \label{Sep.invs}
We have the following bounds:
\begin{enumerate}
\item $\beta_{\rm sep}(R(n,m)) \leq n^6$
\item $\beta_{\rm sep}(S(n,m)) \leq n^6$
\item For any quiver $Q$, $\beta_{\rm sep} (\I(Q,\alpha)) \leq N^6$, where $N = \sum_{i \in Q_0} \alpha_i$.
\item For any quiver $Q$ with no oriented cycles with $|Q_0| = n$, the semi-invariants of weights $\sigma$ such that  $$|\sigma|_\alpha \leq \displaystyle \frac{3}{256} \left(||\alpha||_1^2 - ||\alpha||_2^2\right) \frac{n^2 ||\alpha||_1^{4n}}{(n-1)^{4n-4}},$$ form a separating subset for $\SI(Q,\alpha)$.
\end{enumerate}
\end{corollary}

\section{Computational Invariant theory}
In \cite{Derksen1}, there is a general method for finding degree bounds in characteristic $0$ (see also \cite{DK}). We analyze the role of characteristic $0$ in the method and identify the necessary ingredients in order to adapt the method to positive characteristic. 

Let $V$ be a rational representation of a reductive group $G$. We define the null cone and the Hilbert series.

\begin{definition}
The null cone $\mathcal{N}(G,V) \subseteq V$ is the zero set of all homogeneous invariants in $K[V]^G$ of positive degree. 
\end{definition}

\begin{definition}
For a graded ring $R = \bigoplus_{i \in \N} R_i$, we define its Hilbert series $$H(R,t) = \sum_{i \in \N} \dim(R_i) t^i.$$
\end{definition}

\begin{theorem} [Derksen] \label{D.gen}
Assume $\kar K = 0$, and suppose $f_1,f_2,\dots,f_l$ are homogeneous invariants defining the null cone, and $\deg(f_i) = d_i$. Then we have $$\beta(\K[V]^G) \leq \max\{d_1,d_2,\dots,d_l, d_1 + d_2 + \dots + d_l + \deg(H(K[V]^G,t)\}.$$
\end{theorem}

In the situation above, $H(K[V]^G,t)$ is a rational function, and so we define $\deg(H(K[V]^G,t))$ to be the degree of the numerator $-$ degree of the denominator. The above theorem is a consequence of the fact that the invariant ring $K[V]^G$ is Cohen-Macaulay in characteristic $0$, by the Hochster-Roberts theorem (see \cite{HR}). In order to make effective use of the theorem above to obtain concrete bounds, one needs information about the degree of the Hilbert series as well as a set of invariants defining the null cone. 

In characteristic $0$, Kempf proved that the Hilbert series is a rational function of non-positive degree, see \cite{Kempf}. Knop showed that in certain cases, the degree is in fact atmost the negative of the Krull dimension of the invariant ring, and in some more special cases that the degree is equal to $-\dim V$ (see \cite{Knop1,Knop2,DK}). Further, there is a general method to obtain a set of invariants defining the null cone in characteristic $0$, see \cite{Derksen1}. Hence Theorem~\ref{D.gen} can be used in characteristic $0$ to give concrete bounds. 

The following proposition uses the above techniques to obtain degree bounds.

\begin{proposition} \label{deg.bds.pos}
Assume the following:

\begin{enumerate}
\item $K[V]^G$ is Cohen-Macaulay;
\item $f_1,f_2,\dots,f_l$ are a set of homogeneous invariants defining the null cone, and let $\deg(f_i) = d_i$;
\item $\deg(H(K[V]^G,t)) \leq r$ for some $r \in \Z$. 
\end{enumerate}

Then we have $$\beta(K[V]^G) \leq \max \{d_1,d_2,\dots,d_l, d_1 + d_2 + \dots + d_l + r\}.$$
\end{proposition}

In the cases of interest to us, $(1)$ and $(3)$ will be a consequence of good filtrations, and $(2)$ will be provided by the results in \cite{DM}.

\section{Good Filtrations}
The theory of good filtrations is very powerful in positive characteristic. A comprehensive introduction to this theory can be found in \cite{Donkin2} (see also \cite{Don, Don2, Don3, Hashimoto, Mat}). We also refer the reader to \cite{Domokos2,Zubkov} for an exposition with a view of using them for invariant rings coming from quivers.

Let $G$ be a reductive group, and fix a torus $T$ and fix a borel $B$ containing the torus.  Let $\Lambda$ denote the set of dominant weights. Given $\lambda \in \Lambda$, we have $\lambda:T \rightarrow \K^*$, and we can extend it to a map $\lambda:B \rightarrow \K^*$, by composing with the natural surjection $B \twoheadrightarrow T$. 

\begin{definition}
For $\lambda \in \Lambda$, the dual Weyl module $\nabla(\lambda)$ is defined as 
$$\nabla(\lambda) := \{f \in K[G]\  |\  f(bg) = \lambda(b) f(g)\  \forall \ (b,g) \in B \times G\}.$$
\end{definition}

\begin{definition}
A $G$-module $V$ is called a good $G$-module if $V$ has a filtration of the form $0 \subseteq V_0 \subseteq V_1 \subseteq~\dots$ such that $\bigcup\limits_i V_i = V$ and each quotient $V_i/V_{i-1}$ is a dual Weyl module. Such a filtration is called a good filtration.
\end{definition}

We collect the properties of good filtrations that we require. For the proofs, we refer to the aforementioned references for the theory of good filtrations. 

\begin{proposition} \label{good.filt.prop}
Let $G_1$ and $G_2$ be two reductive groups. Let $V,W$ be good $G_1$-modules. Then we have:
\begin{enumerate}
\item $V \otimes W$ is a good module for $G_1$;
\item $V$ is a good module for $G_1 \times G_2$, where $G_2$ acts trivially on $V$;
\item $\dim(V^{G_1}) = $ multiplicity of the trivial module in any good filtration of $V$;
\item $V$ is a good module for $[G_1,G_1]$;
\item Suppose $W \subseteq V$, then $V/W$ is a good $G_1$-module.
\end{enumerate}
\end{proposition}

We also recall the main result from \cite{Hashimoto}, which will be crucial to our purposes. 

\begin{theorem} [Hashimoto] \label{CMgf}
Assume $\kar K > 0$. Let $V$ be a rational representation of a connected reductive group $G$, and assume its coordinate ring $K[V]$ is a good $G$-module. Then $K[V]^G$ is strongly $F$-regular and hence Cohen-Macaulay.
\end{theorem}

\section{Good filtration of $Rep(Q,\alpha)$}
In this section, we will sketch a proof of the following proposition using arguments that are similar to the ones in \cite{Domokos2,Zubkov}.

\begin{proposition} \label{gfrep}
$K[\Rep(Q,\alpha)]$ is a good $\GL(\alpha)$-module and hence a good $\SL(\alpha)$-module. 
\end{proposition}

First note that since $\SL(\alpha) = [\GL(\alpha),\GL(\alpha)]$, it suffices to prove that $K[\Rep(Q,\alpha)]$ is a good $\GL(\alpha)$-module.

Assume $V$ and $W$ be finite dimensional vector spaces. 
For a partition $\lambda$ such that $\lambda_1 \leq \dim V$, let $L_{\lambda}$ denote the Schur functor corresponding to $\lambda$ as in \cite{ABW}. We note that $L_{\lambda}(V)$ is simply the dual Weyl module for the weight $\widetilde{\lambda}$, the conjugate partition to $\lambda$. In particular, $L_{(r)} (V)$ denotes the $r^{th}$ exterior power $\bigwedge^r(V)$ and 
$L_{(1^r)}(V)$ denotes the $r^{th}$ symmetric power $\Sym^r(V)$.

Part $(1)$ of the following theorem can be found in \cite{ABW} and part $(2)$ in \cite{AB}.

\begin{theorem} [\cite{ABW,AB}] \label{ABWAB}
We have the following:
\begin{enumerate}
\item $\Sym^t(V \otimes W)$ has a filtration whose associated graded object is $$\bigoplus_{|\lambda| = t} L_{\lambda}(V) \otimes L_{\lambda}(W).$$
\item $L_{\lambda}(V)$ has a finite resolution all of whose terms are direct sums of tensors of exterior powers of $V$.
\end{enumerate}
\end{theorem}

We wish to show that $L_{\lambda}(W^*)$ is a good $GL(W)$-module. We follow the argument in \cite{Zubkov}, and first observe that $\bigwedge^i(W^*)$ is a good $GL(W)$-module, since $\bigwedge^i(W^*) \cong \bigwedge^{\dim W - i}(W) \otimes K(det^{-1})$ where $K(det^{-1})$ denotes the $1$-dimensional representation $\GL(V) \rightarrow \GL(K) = K^*$ given by $g \mapsto det(g)^{-1}$. Now, by part $(2)$ of the above theorem, we have that $L_{\lambda}(W^*)$ has a finite resolution whose terms are direct sums of tensors of exterior powers of $W^*$, and hence a finite resolution by good $\GL(W)$-modules. Hence by part $(5)$ of Proposition~\ref{good.filt.prop}, we have that $L_{\lambda}(W^*)$ is a good $\GL(W)$-module. 
 
 By part $(2)$ of Proposition~\ref{good.filt.prop}, we have that $L_{\lambda}(V)$ and $L_{\lambda}(W^*)$ are good $\GL(V) \times \GL(W)$-modules, and hence so is their tensor product. By the filtration given in part $(1)$ of the above theorem, we have that $\Sym^t(V \otimes W^*)$ is a good $\GL(V) \times \GL(W)$-module. Note also that the same argument shows that $\Sym^t(V \otimes V^*)$ is a good $\GL(V)$-module. 
 
 Now, we are ready to prove Proposition~\ref{gfrep}.
 
\begin{proof} [Proof of Proposition~\ref{gfrep}]
For each arrow $a \in Q_1$ such that $ta \neq ha$, we have that $\Sym^t(V(ta) \otimes V(ha)^*)$ is a good $\GL(V(ta)) \times \GL(V(ha))$-module for all $t \in \N$, and hence a good $\GL(\alpha)$-module. For each arrow $a \in Q_1$ such that $ta = ha = i$, we have that $\Sym^t(V(ta) \otimes V(ha)^*) = \Sym^t(V(i) \otimes V(i)^*)$ is a good $\GL(V(i))$-module, and hence a good $GL(\alpha)$-module. Now, 
$$K[\Rep(Q,\alpha)]_d = \bigoplus_{(t_a) \in (\N)^{Q_1},  \sum_{a\in Q_1} t_{a} = d} \left(\bigotimes\limits_{a \in Q_1} \Sym^{t_a}(V(ta) \otimes V(ha)^*)\right),$$

and hence a good $\GL(\alpha)$-module. Hence we have that $K[\Rep(Q,\alpha)]$ is a good $\GL(\alpha)$-module, and consequently a good $\SL(\alpha)$-module.

\end{proof}

\begin{corollary} \label{CM}
The invariant rings $\SI(Q,\alpha),\I(Q,\alpha),S(n,m)$ and $R(n,m)$ are Cohen-Macaulay.
\end{corollary}

\begin{proof}
In characteristic $0$, this follows from the Hochster-Roberts theorem (see \cite{HR}), and in positive characteristic, this follows from Theorem~\ref{CMgf}.
\end{proof}

\section{Hilbert series}
The aim of this section is to show that the Hilbert series of the invariant rings of interest to us do not depend on the characteristic. 
\begin{proposition} \label{Hilb.si}
The Hilbert series  $H(\SI(Q,\alpha),t)$ and $H(\I(Q,\alpha),t)$ are independent of the underlying field $K$.
\end{proposition}

\begin{proof}
We only show that the $H(\SI(Q,\alpha),t)$ is independent of the underlying field $K$, as the argument for $H(\I(Q,\alpha),t)$ is similar. 

It suffices to show that the filtration multiplicity of the trivial module in $K[\Rep(Q,\alpha)]_d$ is independent of the underlying field by part $(3)$ of Proposition~\ref{good.filt.prop}. Indeed, observe that the character of $K[\Rep(Q,\alpha)]_d$ is independent of the underlying field. Further, the characters of the dual Weyl modules are given by the Weyl character formula (see \cite{Donkin2}), and hence independent of the underlying field. The characters of the dual Weyl modules are linearly independent, since it is true over $\C$. We can write the character of $K[\Rep(Q,\alpha)]_d$ as a $\N$-linear combination of the characters of the dual Weyl modules since $K[\Rep(Q,\alpha)]_d$ has a good filtration. The multiplicities of the dual Weyl modules are given by the coefficients of this linear combination, which is clearly independent of the underlying field. 

\end{proof}

Applying the above Proposition to the $m$-Kronecker quiver and the $m$-loop quiver, we get the following:

\begin{corollary} \label{Hilb.msi}
The Hilbert series $H(R(n,m),t)$ and $H(S(n,m),t)$ are independent of the underlying field $K$. 
\end{corollary}

\section{Matrix invariants}
Procesi showed that the invariant ring $S(n,m)$ is generated by traces of monomials in characteristic $0$, see \cite{Procesi}. Using the theory of good filtrations, Donkin extended Procesi's result to all characteristics (see \cite{Donkin}) by taking characteristic coefficients rather than traces. For an $n \times n$ matrix $A$, we have $\det(tI_n - A) = \sum_{j \in \N} \sigma_j(A) t^j$.  $\sigma_j$ is called the $j^{th}$ characteristic coefficient of $A$.

\begin{theorem} [Donkin]
The ring $S(n,m)$ is generated by $\sigma_s (X_{i_1}X_{i_2}\dots X_{i_k})$, $s \geq 0$.
\end{theorem} 

Along with Le~Bruyn, Procesi extended his result to invariants of quivers in characteristic $0$ in \cite{LP}, and Donkin did the same for his result in arbitrary characteristic in \cite{Donkin3}.

\begin{theorem} [Donkin]
The ring $\I(Q,\alpha)$ is generated by $$\{ \sigma_j(V(a_n)V(a_{n-1}) \dots V(a_1))\  |\   j \geq 0, \  a_1a_2\dots a_n \text{ oriented cycle in } Q\}.$$ 
\end{theorem}

Define $V = \bigoplus_{i \in Q_0} V(i)$. For each arrow $a \in Q_1$, we have a natural embedding
 $$\Hom(V(ta),V(ha)) \hookrightarrow \Hom(V,V).$$
  Putting together the maps for each arrow, we get
 $$\Rep(Q,\alpha) \hookrightarrow \Hom(V,V)^{\oplus M},$$ 
where $M = |Q_1|$. In short, the above two results of Donkin give us the following:

\begin{corollary} \label{mi.iq}
There is a surjection of graded rings $S(N,M) \twoheadrightarrow I(Q,\alpha)$, where $N = \sum_{i \in Q_0} \alpha_i$ and $M = |Q_1|$.
\end{corollary}

\section{Semi-invariants}
For semi-invariants, there exist determinantal descriptions of the invariants (see \cite{DW,DZ,SV}), but we will not recall them as we do not need them explicitly. We refer to \cite{DM} for details consistent with our notation. We will however recall the following crucial result from \cite{DM}.

\begin{theorem} [Derksen-Makam] \label{nullcone.bound}
Assume $n \geq 2$. Let $r =$ Krull dimension of $R(n,m)$. Then the null cone $\mathcal{N}(SL_n \times SL_n, \M_{n,n}^m)$ is defined by $r$ invariants of degree $n(n-1)$.
\end{theorem}

We also recall a result of Domokos in \cite{Domokos1} which allows us to give bounds for invariants using the bounds on semi-invariants. 

Consider the injection $\phi: \M_{n,n}^m \rightarrow \M_{n,n}^{m+1}$ given by $(X_1,X_2,\dots,X_m) \mapsto (I,X_1,X_2,\dots,X_m)$. This gives a surjective map on the coordinate rings $\phi^* :K[\M_{n,n}^{m+1}] \rightarrow K[\M_{n,n}^m]$. More precisely, for $f \in K[\M_{n,n}^m]$, we have $\phi^* (f) (X_1,X_2,\dots,X_m) = f(I,X_1,\dots,X_m)$. Domokos showed that $\phi^*$ descends to a map on invariant rings as given below:

\begin{proposition} [Domokos] \label{msi.mi}
We have a surjection $\phi^* : R(n,m+1) \twoheadrightarrow S(n,m)$.
\end{proposition}

\begin{corollary} \label{something}
We have $\beta(S(n,m)) \leq \beta(R(n,m+1))$.
\end{corollary}

\begin{proof}
This follows since $\deg(\phi^*(f)) \leq \deg(f)$ for all $f \in R(n,m+1)$. 
\end{proof}

\section{Proofs of main results}
The various tools developed in the previous sections fall into place to give our main results.

\begin{proof} [Proof of Theorem~\ref{msi}]
For $n = 1$, $\SL_1 \times \SL_1$ is trivial, and hence $R(1,m)$ is a polynomial ring in $m$ variables. Hence the result holds. 

Now, assume $n \geq 2$. We have that $R(n,m)$ is Cohen-Macaulay, by Corollary~\ref{CM}. We also have that $\deg(H(R(n,m),t) \leq 0$ in characteristic $0$ by Kempf's result (see \cite{Kempf}) and hence in any characteristic by Corollary~\ref{Hilb.msi}. In fact Knop's stronger bounds also apply, but we will not need them. By Theorem~\ref{nullcone.bound}, we have $r$ invariants of degree $n(n-1)$ that define the null cone. Further $r \leq \dim \M_{n,n}^m = mn^2$, so by Proposition~\ref{deg.bds.pos}, we have 
$$\beta(R(n,m)) \leq r n(n-1) \leq mn^3(n-1) \leq mn^4.$$
\end{proof}

\begin{remark} \label{stronger}
Note that the proof actually gives $\beta(R(n,m)) \leq mn^3(n-1)$ for $n \geq 2$. We will use this stronger bound to get the bounds for separating invariants for $S(n,m)$.
\end{remark}

\begin{proof} [Proofs of Corollary~\ref{si1} and Theorem~\ref{si2}]
The invariant rings under consideration are Cohen-Macaulay by Corollary~\ref{CM}, and that the degree of the Hilbert series is independent of the underlying field by Proposition~\ref{Hilb.si} and Corollary~\ref{Hilb.msi}. Also note that the results on the null cone in \cite{DM,DM2} were independent of the underlying field $K$. Hence the arguments in \cite{DM,DM2} work in arbitrary characteristic by replacing Theorem~\ref{D.gen} with Proposition~\ref{deg.bds.pos} wherever necessary.
\end{proof}

\begin{proof} [Proof of Theorem~\ref{mi.bds}]
This follows from Corollary~\ref{something}.
\end{proof}

\begin{proof}[Proof of Corollary~\ref{inv.bds}]
This follows from the surjection in Corollary~\ref{mi.iq}
\end{proof}

\begin{proof} [Proof of Corollary~\ref{Sep.invs}]
Observe that a set of generating invariants is a separating subset, and hence we have 
$$\beta_{\rm sep} (K[V^{\oplus n}]^G) \leq \beta_{\rm sep}(\K[V^{\oplus \dim V}]^G) \leq \beta (\K[V^{\oplus \dim V}]^G).$$ 

Now for parts $(1),(3)$ and $(4)$, the bounds for separating invariants follow from the corresponding bounds on generating invariants shown in this paper. 

For part $(2)$, we use the slightly stronger bounds in Remark~\ref{stronger}. So for $n \geq 2$, we get:

$$\beta_{\rm sep} (S(n,m)) \leq \beta_{\rm sep} (S(n,n^2)) \leq \beta(S(n,n^2)) \leq \beta(R(n,n^2+1)) \leq (n^2 +1)n^3(n-1) \leq n^6.$$

For $n = 1$, observe that $S(1,m) = \K[K^m]^{\GL_1} = K[K^m]$ i.e., a polynomial ring in $m$ variables,  and hence the bound follows for $n =1$ as well.

\end{proof}

\subsection*{Acknowledgements}
The second author would like to thank Alexandr N. Zubkov for helpful discussions.

\end{document}